\documentclass[10pt]{article}
\usepackage[margin=.72in]{geometry}
\usepackage{amsmath,amssymb,amsthm}
\usepackage{microtype}
\newtheorem{theorem}{Theorem}[section]
\newtheorem{lemma}[theorem]{Lemma}
\newtheorem{corollary}[theorem]{Corollary}
\newcommand{\Pp}{\mathbb P}
\newcommand{\Ee}{\mathbb E}

\renewcommand{\epsilon}{\varepsilon}

\usepackage{graphicx} 
\usepackage[linktocpage=true]{hyperref}
\usepackage{setspace}
\usepackage{tikz}
\usetikzlibrary{decorations.pathreplacing}
\usetikzlibrary{calc}
\usepackage{mathrsfs}
\usepackage{caption,cite}
\usepackage{subcaption}
\usepackage{mathtools,color}
\usepackage{todonotes}
\usepackage{cleveref}

\title{A short proof that $R(3,k)=\Theta(k^2/\log k)$}
\author{ Samuel Harris\thanks{Department of Mathematical and Statistical Sciences, University of Colorado Denver, Denver, USA. E-mail: {\tt Samuel.G2.Harris@ucdenver.edu}}
\and Zion Hefty\thanks{Department of Mathematics, University of Denver, Denver, USA. Email: {\tt Zion.Hefty@du.edu}}
\and  Paul Horn\thanks{Department of Mathematics, University of Denver, Denver, USA and Department of Mathematics and Applied Mathematics, University of Johannesburg, Johannesburg, South Africa. Email: {\tt Paul.Horn@du.edu}.  Research is partially supported by Simons TSM grant 00025233.   }.  
\and Dylan King\thanks{California Institute of Technology, Pasadena, USA. Email: {\tt dking@caltech.edu}} 
\and Florian Pfender\thanks{Department of Mathematical and Statistical Sciences, University of Colorado Denver, Denver, USA. E-mail: {\tt Florian.Pfender@ucdenver.edu}. Research is partially supported by NSF DMS-2152498 and  DMS-2554130.}}
\date{}

\begin{document}
\maketitle
\vspace{-2em}
\begin{center}
    \textit{Dedicated to Micha{\l} Karo\'nski on the occasion of his eightieth birthday}
\end{center}
\begin{abstract}
We give a nibble-free construction proving
$R(3,k)\ge(1/200+o(1))k^2/\log k$.  We also include Shearer's proof bounding the independence number of a triangle-free graph, 
which implies $R(3,k)\le (1+o(1))(k^2/\log k)$.
\end{abstract}

\section{Introduction}

The Ramsey number $R(\ell,k)$ is the least $N$ such that every $N$-vertex graph contains either a clique of size $\ell$ or an independent set of size $k$, and is guaranteed to exist by Ramsey's Theorem~\cite{Ramsey1930}. Recent years have seen a number of breakthrough results on Ramsey numbers; see the wonderful survey by Morris~\cite{morrisICM} presented at the International Congress of Mathematicians 2026. Among off-diagonal Ramsey numbers with $\ell$ constant, the first non-trivial case $R(3,k)$ is particularly well studied. For a detailed history, see Spencer's book chapter~\cite{Spencer2011} and the beautiful introduction of a recent paper by Campos, Jenssen, Michelen and Sahasrabudhe~\cite{campos2025}.

In the early 1980s, Ajtai, Koml\'os and Szemer\'edi \cite{AjtaiKSz80,AjtaiKSz81} proved that 
$ R(3,k)=O\left(k^2/\log k\right)$, 
and shortly afterwards Shearer \cite{Shearer83} improved their lower bound on the independence number of triangle-free graphs, which implies that 
\begin{align*}
    R(3,k)\le (1+o(1))\frac{k^2}{\log k}.
\end{align*}
The most successful method for constructing triangle-free graphs with small independence number, and in turn lower bounds on $R(3,k)$, is to create a triangle-free graph of density $p$ that has pseudo-random properties such that its independence number approximates that of a random graph in $G(n,p)$. The larger one can push this $p$, the smaller the resulting independence number.

One way to do this is to start with a binomial random graph sampled from $G(n,p)$ and then delete an edge from every triangle. For this to be successful $p$ must be chosen small enough so that the deleted edges do not create large independent sets, and thus this method is limited to~$p$ where there are significantly fewer triangles than edges. This approach was pioneered by Erd\H{o}s~\cite{Erdos1961_GraphTheoryProbabilityII}
yielding a bound of $R(3,k)\ge (c+o(1)) k^2/\log^2k$.

In 1995, Kim~\cite{Kim95} matched the order of the upper bound and showed that $ R(3,k)\ge (c+o(1))k^2/\log k$ using the opposite approach. Start with an empty graph, and carefully add random edges without creating triangles (a nibble process) until the graph is both pseudo-random and as dense as possible.

In~\cite{BohmanKeevash21, Fiz20, campos2025}, three groups of authors improved on the constant in the lower bound, culminating with $c=\frac13$. Each of these constructions involved some sort of nibble process with a very technical analysis to control the random process over many steps.

Our approach in~\cite{RamseyR3k25} returns to the idea of deleting edges from triangles. The novel idea is to start with a denser random construction in which we can delete edges more efficiently to destroy all triangles. We use blow-ups of random graphs very similar to the one used by Campos, Jenssen, Michelen, and Sahasrabudhe~\cite{campos2025}  at the start of their nibble.
The random union of two such graphs yields a multigraph where every deleted edge destroys about $\log^2n$ triangles, and thus we can afford to start with many more triangles (and therefore higher density~$p$). Further, the random union along with the random initial input creates enough pseudo-randomness to attain the independence bound of the random graph of the same density. This yields the best currently known lower bound for $R(3,k)$.
\begin{theorem}\label{main1/2}\cite{RamseyR3k25}
    \[
R(3,k)\ge  \left(\frac12+o(1)\right)\frac{k^2}{\log{k}}.
\]
\end{theorem}
This bound is the best one can achieve in any pseudo-random construction, as this parallels a random graph where the density is chosen so that the independence number matches the maximum degree. For this reason, Campos, Jenssen, Michelen, and Sahasrabudhe~\cite{campos2025} have conjectured that $R(3,k)= \left(\frac12+o(1)\right)k^2/\log{k}$.
While the construction in \cite{RamseyR3k25} is much easier to analyze, the paper is still somewhat technical and spans about 20 pages in order to achieve this optimal constant.

In the epilogue to his survey chapter on Ramsey theory \cite{Spencer2011} in Soifer's book \emph{Ramsey Theory} \cite{SoiferRamsey}, Spencer wrote already aware of the results in~\cite{Fiz20, BohmanKeevash21}: 

\vspace{0.2cm}
{\em 
``Is the story of $R(3,k)$ over? I think not. I think there is plenty of room for a consolidation of the results. My dream is a ten-page paper which gives $R(3,k)=\Theta(k^2/\log{k})$."
}
\vspace{0.2cm}

In this note, we realize Spencer's stated dream by further simplifying the proof of Theorem~\ref{main1/2} at the cost of a much smaller constant.

\begin{theorem}\label{thm:lower} %
\[
    R(3,k)\ge  \left(\frac1{200}+o(1)\right)\frac{k^2}{\log{k}}.
\]
\end{theorem}

The construction in this note is almost identical to the one in~\cite{RamseyR3k25}. We merely change a few minor details for better readability of the simplified proof.
Compared to~\cite{RamseyR3k25}, our main simplifications occur in two places. In Section~\ref{edgedeletions}, we consider three instead of four classes of vertices. In Section~\ref{projections}, we focus on only one of the two edge colors when bounding the probability that a set is independent. This last simplification, in particular, avoids many technicalities.

As in all previous results, to establish Theorem~\ref{thm:lower} we prove the following equivalent theorem concerning independence numbers. 

\begin{theorem}\label{thm:alpha}
For every $n$, there exists a triangle-free graph $G$ on $n$ vertices with independence number 
    \[
\alpha(G) \le (10+o(1))\sqrt{n \log n}.
\]
\end{theorem}
With a bit of care, our proof here shows $\alpha(G) \le (8+o(1))\sqrt{n \log n}$, a minor improvement we omit for the sake of exposition.
In order to contain a full proof of $R(3,k)=\Theta(k^2/\log k)$, we include Shearer's short proof of the upper bound in the final section.

\section{The random construction of $G$}

Let us now formally define our construction of the desired $n$-vertex graph $G$.  For the sake of readability, we omit floors and ceilings whenever they are inconsequential.

Set
\[
 m=\frac{n}{\log^2 n},\qquad p=\sqrt{\frac{\log n}{n}}, \qquad K = 10pn=10\sqrt{n \log{n}}.
\]
The construction of $G$ 
then consists of the following steps:
\begin{enumerate} 
\item Sample two independent copies of the random graph $G(m,p)$, $G_{R}$ and $G_{B}$, on the vertex set $V_R = V_B=[m]$. 

\item For every vertex $v\in V(G)$, sample independently uniformly at random two values $\pi_R(v),\pi_B(v)\in [m]$.

\item Define a two-colored multigraph $G_1$ on the vertex set $V(G_1)=V(G)$ as follows. If $\pi_R(v)\pi_R(w)\in E(G_R)$, then $vw$ is a red edge in $G_1$. Similarly, if $\pi_B(v)\pi_B(w)\in E(G_B)$, then $vw$ is a blue edge in $G_1$. 

\item Next, define a triangle-free multigraph $G_2\subseteq G_1$.
From every monochromatic red triangle $uvw\subseteq G_1$ with $\pi_R(u)<\pi_R(v)<\pi_R(w)$, remove the red edge $vw$. From every monochromatic blue triangle $uvw\subseteq G_1$ with $\pi_B(u)<\pi_B(v)<\pi_B(w)$, remove the blue edge $vw$.

Finally, 
remove the edge of the minority color in every non-monochromatic triangle in $G_1$. 
If a triangle $uvw$ contains double edges, then we consider the resulting triangles separately. 

\item The graph $G$ is then the simple graph underlying $G_2$. 
\end{enumerate} 

For $x\in V_R$ and $y\in V_B$, let $F(x)=\pi_R^{-1}(x)$ and $F(y)=\pi_B^{-1}(y)$ denote the fibers in $V(G)$.
Observe that the monochromatic subgraphs of $G_1$ are blow-ups of $G_R$ and $G_B$, where each vertex $x$ is blown up to an independent set of size $|F(x)|$, and the vertices of the two blow-up graphs are identified at random.

\section{Edge deletions}\label{edgedeletions}

In this section, we prove two lemmas that allow us to control the number of deleted edges in $K$-sets in $G$.

If an edge $vw$ in a triangle $uvw\subset G_1$ is deleted, then $uv$ and $uw$ have the same color. For this reason, we want to study the monochromatic neighborhoods of vertices $u$. Observe that if $x\in V_R$ and $u\in F(x)$, then all vertices in $F(x)$ have the same red neighborhood $X_x:=\bigcup_{xz\in E(G_R)} F(z)$. Analogously, define $X_y=\bigcup_{yz\in E(G_B)} F(z)$ for $y\in V_B$.

Chernoff (see for instance~\cite{FriezeKarBook}, Chapter 21) and union bounds show the following event. In all that follows, let $0<\epsilon\le \frac1{1000}$ be a universal constant. 

\begin{lemma}\label{lem:fiber_and_degree}
With probability $1-o(1)$, the following hold for all distinct $x,y \in V_R \cup V_B$.
\begin{align}
 |F(x)|&\le (1+\epsilon)\log^2 n,\label{eq:fiber}\\
 |X_x|&\le(1+\epsilon)pn,\label{eq:degree}\\
 |X_x\cap X_y|&\le \log^3n. 
 \label{eq:codegree}
\end{align}
\end{lemma}

\begin{proof}[Proof of Lemma~\ref{lem:fiber_and_degree}]
Chernoff and union bounds on the random graph $G_R$ give that with probability $1-o(1)$, for all distinct $x,y\in V_R$, $|N_{G_R}(x)|<(1+\epsilon^2)pm$ and $|N_{G_R}(x)\cap N_{G_R}(y)|<0.5\log n$, and analogous bounds for $G_B$. Now fix $G_R$ and $G_B$ and assume that the previous bounds hold.

For $x\in V_R$, the fiber size $|F(x)|$ is $\operatorname{Bin}(n,1/m)$, and the quantity $|X_x|$ is $\operatorname{Bin}(n,|N_{G_R}(x)|/m)$. If $y\in V_R$, then $|X_x\cap X_y|$ is $\operatorname{Bin}(n,|N_{G_R}(x)\cap N_{G_R}(y)|/m)$. If $y\in V_B$, then $|X_x\cap X_y|$ is $\operatorname{Bin}(n,|N_{G_R}(x)|| N_{G_B}(y)|/m^2)$. The symmetric statements are true for $x\in V_B$. The Lemma now follows from Chernoff and union bounds over all choices of $x$ and $y$.
\end{proof}

For $I \subseteq V(G)$ with $|I| = K$, let $I_R=\cup_{v \in I}\pi_R(v)$ and $I_B=\cup_{v \in I}\pi_B(v)$ denote the projection of $I$ onto $G_R$ and $G_B$ respectively. 
Set $t_1=n^{1/4+\epsilon}$ and $t_2=n^{2\epsilon}$, and partition the vertices of $V_R \cup V_B$ by the size of $X_x\cap I$:
\[
 L_I=\{x:|X_x\cap I|>t_1\},\quad
 M_I=\{x:t_2<|X_x\cap I|\le t_1\},\quad
 S_I=\{x:|X_x\cap I|\le t_2\}.
\]
The next lemma addresses the total contribution of pairs in the $X_x\cap I$. 
\begin{lemma}\label{lem:closed}
With probability $1-o(1)$, simultaneously for all $I \subseteq V(G)$ with $|I| = K$,
\begin{align}\label{eq:closed}
 \sum_{x\in V_R\cup V_B} {|X_x\cap I|\choose 2}<\left(\frac{1}{20}+5\epsilon\right)K^2.
\end{align}
\end{lemma}

\begin{proof}
We start by bounding pairs inside fibers $F(x)$. For each $x\in I_R$, there are at most $|F(x)|\le (1+\epsilon)\log^2n$ vertices in $I\cap F(x)$ by~\eqref{eq:fiber}. Thus, the number of such pairs is maximized at less than $\frac12(1+\epsilon)K\log^2n$ by setting $|I_R|=K/((1+\epsilon)\log^2n)$. The same bound holds for $x\in I_B$ for a total of less than $(1+\epsilon)K\log^2n<\epsilon K^2$ pairs inside fibers.

For all remaining pairs, we treat the three classes separately, and always assume that $n$ is large enough for the stated inequalities to hold.  First, $|L_I|<2K/t_1$.  Otherwise inclusion-exclusion and \eqref{eq:codegree}, applied to $2K/t_1$ members of $L_I$, give a union of size greater than $K$, since $K\log^3n/t_1^2=o(1)$.  Another inclusion-exclusion estimate, applied to the entire family $L_I$, gives 
\[ 
\sum_{x\in L_I}|X_x\cap I| \le K+\binom{|L_I|}{2}\log^3 n < K+\frac{2K^2}{t_1^2} \log^3 n<(1+\epsilon) K. 
\]

Hence, using \eqref{eq:degree} and $K=10pn$,
\[
 \sum_{x\in L_I}\binom{|X_x \cap I|}{2}
 <\frac{(1+\epsilon)pn}{2}\sum_{x\in L_I}|X_x \cap I|
 <\frac{(1+\epsilon)^2}{20}K^2 < \left(\frac{1}{20}+\epsilon\right)K^2.
\]

For the vertices in $M_I$, we show that, with high probability, every $K$-set
$I\subseteq V(G)$ satisfies

\begin{align}\label{eq:medium-sum}
 \sum_{x\in M_I}|X_x \cap I|\le Kn^{1/4-\epsilon}.
\end{align}
Indeed, 
this gives
\[
 \sum_{x\in M_I}\binom{|X_x\cap I|}2
 \le \frac{t_2}{2}\sum_{x\in M_I}|X_x \cap I|
 \le \frac{K n^{1/2}}2
 <\epsilon K^2.
\]
If \eqref{eq:medium-sum} fails, let $B \subseteq M_I$ be minimal so that $Y = \sum_{x \in B} |X_x\cap I| \ge Kn^{1/4 - \varepsilon}$, and in particular $Y\le Kn^{1/4 - \varepsilon}+t_1$.  Note that $|B| \leq n^{-2\varepsilon}Y < 2Kn^{1/4 - 3\varepsilon}.$  

Write $B_R=B\cap V_R$ and $B_B=B\cap V_B$.
On \eqref{eq:fiber}, every edge
between $B_R$ and $I_R$ in $G_R$, or between $B_B$ and $I_B$
in $G_B$, accounts for at most $2(1+\epsilon)\log^2 n$ incidences counted by $Y$.
Consequently, the graphs $G_R$ and $G_B$ contain together at least
\[
 q:=\frac{Y}{2(1+\epsilon)\log^2 n}
  >\frac{K n^{1/4-\varepsilon}}{2(1+\epsilon)\log^2 n}
\]
edges between $B_R$ and $I_R$, and $B_B$ and $I_B$, respectively.

For fixed $I$ and fixed $B$, the number $Z$ of these 
edges is stochastically dominated by
$ \operatorname{Bin}\!\left(2K|B|,p\right)$ with mean $
 2K|B|p\le 2K^2n^{1/4-3\varepsilon}p=o(q)$.
Thus, for sufficiently large $n$, a Chernoff bound gives that
$
 \Pp(Z\ge q)
 \le \exp\!\left(-n^{3/4-2\varepsilon}\right).
$

There are at most 
\[
 \binom nK
 \sum_{j\le 2K n^{1/4-3\varepsilon}}\binom{2m}{j} \leq \left(\frac{en}{K}\right)^K \left(2Kn^{1/4-3\epsilon}\right) \left( \frac{em}{Kn^{1/4-3\epsilon}} \right)^{2Kn^{1/4-3\epsilon}}
 \leq \exp\left(\frac12 n^{3/4-2\epsilon}\right)
\]
choices for $(I,B)$.  A union bound therefore shows that, with
probability $1-o(1)$, no such pair $(I,B)$ exists.

Finally consider vertices in $S_I$.  Vertices  $x\in S_I\cap (I_R\cup I_B)$ contribute at most
$2K\binom{t_2}{2}<\epsilon K^2$ to the sum in~\eqref{eq:closed}.  Let 
\[
Z_I=\sum_{x\in S_I\setminus(I_R\cup I_B)}{|X_x\cap I|\choose 2}
\]
count the remaining pairs, and let $Z_I^*\le Z_I$ only count pairs which are not inside some fiber (we have accounted for such pairs at the start of the proof).  
Every pair of vertices in $G_1$ that is not in some fiber $F(x)$ has probability at most 
$2mp^2=2/\log n$ to end up in some $X_x$ when $G_R$ and $G_B$ are exposed, so $\Ee[ Z_I^*]\le K^2/\log n$. 
 Reveal the edges in $G_R$ to expose $X_x\cap I$ for all $x\in V_R\setminus I_R$ step by step in order of $x$, followed by the edges in $G_B$ for $x\in V_B\setminus I_B$.  Changing one step and thus one $X_x\cap I$ changes $Z_I^*$ by at most $\binom{t_2}{2}$.  McDiarmid's inequality (see~\cite{FriezeKarBook}, Chapter 21) gives
\[
 \Pp\left(Z_I^*>\epsilon K^2\right)\le \Pp\left(Z_I^*>K^2/\sqrt{\log n}\right)
 \le\exp\bigl(-10^4n^{1-8\epsilon}\log^3n\bigr)
 =o\left(\binom nK^{-1}\right).
\]
A union bound over all $I$ shows that with probability $1-o(1)$, we have $Z_I^*<\epsilon K^2$ for all $I$. Adding all bounds for the various pairs finishes the proof of the lemma.
\end{proof}

\section{Size of projections and ordered exposure}\label{projections}

We now turn to the crucial part of the proof of Theorem~\ref{thm:alpha}, showing that vertex sets of size $K$ are not independent.  
We first show, with high probability, a lower bound on the size of the maximum of $|I_R|$ and $|I_B|$. Here and later, we use that $\binom nK=\exp\left(\left(\frac12+o(1)\right)K\log n\right)$.
\begin{lemma}\label{lem:projection}
With probability $1-o(1)$, 
$\max\{|I_R|,|I_B|\}\ge(\frac12-\epsilon)K$ for all $I \subseteq V(G)$ with $|I| = K$.
\end{lemma}

\begin{proof}
For fixed $I$ and $0<\gamma<1$,
\[
 \Pp(|I_R|\le\gamma K)
 \le\binom m{\gamma K}\left(\frac{\gamma K}{m}\right)^K
 =\exp\left(-(1-\gamma+o(1))K\log(m/K)\right)
 =\exp\left(-\tfrac12(1-\gamma+o(1))K\log n\right).
\]
The random variables $|I_R|$ and $|I_B|$ are independent.
Taking the union bound over all $K$-sets $I$ and setting $\gamma=\frac12-\epsilon$, we then have
\begin{align*}
    \Pp\left(\exists I:(|I_R|\le (\tfrac12-\epsilon) K)\land (|I_B|
    \le(\tfrac12-\epsilon) K)\right)
 \le
 {n\choose K}\exp\left(-(\tfrac12+\epsilon)K\log n\right)
 =\exp\left((-\epsilon+o(1))K\log n\right).
\end{align*}
\end{proof}
For a fixed $I$, define the event $\pi^I_{r,b}:=(|I_R|=r)\land (|I_B|=b)$.
\begin{lemma}\label{lem:exposure}
Fix $I$ and let $r,b\in [1,K]$ with 
$\max\{r,b\}\ge(\frac12-\epsilon)K$.  On the event \eqref{eq:closed} and $n$ large enough,
\[
 \Pp(I\text{ is independent in }G) \mid \pi^I_{r,b})
 \le\exp\left(-p\left(\frac{\left(\frac12-\epsilon\right)^2}{2}-\frac{1}{20}-6\epsilon\right)K^2\right).
\]
\end{lemma}
\begin{proof}
Assume $r\ge(\frac12-\epsilon)K$, and completely reveal $\pi_R$ and $\pi_B$.  Then reveal all pairs of $G_B$, and finally expose the pairs in $G_R$ in lexicographic order. This order guarantees that $X_x$ is determined for all vertices $x\in V_B$ and for all vertices $x\in V_R$ with $x<\pi_R(v),\pi_R(w)$ before the Bernoulli variable determining the pair $\pi_R(v)\pi_R(w)$ is seen.

If $I$ is independent in $G$, $\{v,w\}\subset I$, and $\{v,w\}\not\subset X_x$ for all vertices $x\in V_B$ and for all vertices $x\in V_R$ with $x<\pi_R(v),\pi_R(w)$, we have $\pi_R(v)\pi_R(w)\notin E(G_R)$, for otherwise the red edge $vw\in E(G_1)$ survives both deletion rules.  
Lemma~\ref{lem:closed} and sequential conditioning gives
\begin{align*}
\Pp(I\text{ is independent in }G \mid \pi^I_{r,b})
 \le(1-p)^{\binom{r}{2}-(1/20+5\epsilon)K^2}\le \exp\left(-p\left(\tfrac{\left(\frac12-\epsilon\right)^2}{2}-\tfrac{1}{20}-6\epsilon\right)K^2\right).
\end{align*}
\end{proof}

The proof of Theorem~\ref{thm:alpha} and thus Theorem~\ref{thm:lower} follows easily.
The bad events in Lemmas~\ref{lem:closed} and \ref{lem:projection} have probability $o(1)$.  Otherwise, Lemma~\ref{lem:exposure} applies.  Therefore, by taking a union bound over all $I$
and using that $pK=10\log n$,
\begin{align*}
    \Pp(\alpha(G)\ge K)
    &\le o(1)+{n\choose K}\max_{\substack{r,b\\r\ge (\frac12-\epsilon)K}}\Pp(I\text{ is independent in }G \mid \pi^I_{r,b})\\
    &\le o(1)+{n\choose K}\exp\left(-p\left(\frac{\left(\frac12-\epsilon\right)^2}{2}-\frac{1}{20}-6\epsilon\right)K^2\right)\\
   & \le o(1)+\exp\left(\left[\frac12+o(1)-10\left(\frac{\left(\frac12-\epsilon\right)^2}{2}-\frac{1}{20}-6\epsilon\right)\right]K\log n\right)\\
   &\le o(1)+\exp\left(-\frac1{10}K\log n\right).
\end{align*}
\qed

\section{Shearer's upper bound}\label{upper}
We now include a proof of Shearer's bound following the presentation in~\cite{Shearer91}.  Define
\begin{align}\label{eq:shearer-recursion}
 f(0)=1,\qquad
 f(d)=\frac{1+(d^2-d)f(d-1)}{d^2+1}\quad(d\ge1).
\end{align}
A direct induction from the recursion shows that $f(d)$ and the difference
$f(d)-f(d+1)$ are both non-increasing. 

\begin{theorem}\cite{Shearer91}\label{thm:shearer}
If $G$ is a triangle-free graph on $n$ vertices with degree sequence $d_1,\ldots,d_n$, then
\[
 \alpha(G)\ge\sum_{i=1}^n f(d_i).
\]
\end{theorem}

\begin{proof}
We will prove the theorem by induction on $n$, and note that it holds for $n=0$. Let $S(G) =\sum_{i=1}^nf(d_i)$.  For a vertex $i$, let
$N_1^i=N(i)$, and let $N_2^{i}$ be the vertices at distance two from $i$.  For all $k \in N_1^i$, let $n_i(k) = |N_1^k \cap N_1^i|$.  Define $H_i := G-\big(\{i\} \cup N_1^i\big)$. 
Then triangle-freeness gives
\begin{align*}
 S(H_i)=S(G)-f(d_i)-\sum_{j\in N_1^i}f(d_j)
 +\sum_{k\in N_2^i}\bigl(f(d_k-n_i(k))-f(d_k)\bigr).
\end{align*}

If some $i$ satisfies
\begin{align}\label{eq:good-i}
 1-f(d_i)-\sum_{j\in N_1^i}f(d_j)
 +\sum_{k\in N_2^i}\bigl(f(d_k-n_i(k))-f(d_k)\bigr)\ge0,
\end{align}
then induction shows the result, as a maximum independent set of $H_i$, together with $i$, has size at least $S(G)$.

We show that \eqref{eq:good-i} holds on average.  Let $A$ be the sum of the left side of  \eqref{eq:good-i} over all $i$, and let $ B_i=\sum_{k\in N_2^i}[f(d_k-n_i(k))-f(d_k)]$.  Changing the order of summation gives
\[
 A=\sum_i\left(1-(d_i+1)f(d_i)+B_i\right).
\]
Since $f(d-1)-f(d)$ is a non-increasing function of $d$,
\[
 f(d_k-n_i(k))-f(d_k)
 \ge n_i(k)[f(d_k-1)-f(d_k)].
\]
Summing over $i$ and pairing the two orientations of each edge gives
\begin{align*}
 \sum_iB_i\ge
 \sum_i(d_i^2-d_i)[f(d_i-1)-f(d_i)].
\end{align*}
The edgewise comparison used here is:
\[
 (d_i-d_j)\bigl([f(d_j-1)-f(d_j)]-[f(d_i-1)-f(d_i)]\bigr)\ge0.
\]
Consequently,
\begin{align*}
 A&\ge\sum_i\bigl(1-(d_i+1)f(d_i)
 +(d_i^2-d_i)(f(d_i-1)-f(d_i))\bigr)\\
 &=\sum_i\bigl(1+(d_i^2-d_i)f(d_i-1)-(d_i^2+1)f(d_i)\bigr)=0,
\end{align*}
where the final equality is \eqref{eq:shearer-recursion}.  Hence some $i$ satisfies \eqref{eq:good-i}, completing the induction.
\end{proof}

\begin{corollary}\label{cor:upper}
\[
 R(3,k)\le\left(1+o(1)\right)\frac{k^2}{\log k}.
\]
\end{corollary}
\begin{proof}
    Again, it suffices to show a bound on the independence number of triangle-free graphs on $n$ vertices.
    The neighborhood of any vertex in a triangle-free graph is independent. Denoting by $\Delta$ the maximum degree of $G$, we have from Theorem~\ref{thm:shearer} (noting again that $f$ is non-increasing) that
    $\alpha(G)\ge\max\{\Delta, nf(\Delta)\}$.
    
    For fixed $n$, this bound is minimized for $\Delta=nf(\Delta)$, so we may assume that this equality holds. The recursion~\eqref{eq:shearer-recursion} implies that $\left|df(d)-(d-1)f(d-1)- \frac{1}{d}\right|<\frac{1}{d^{2}}$, and therefore $df(d) = \log{d} + O(1)$. Thus, 
    \[
    n=\frac{\Delta}{f({\Delta})}=\frac{\Delta^2}{\log\Delta+O(1)},
    \]
    implying the desired bound
    \[
     \alpha(G)\ge \Delta=(1+o(1))\sqrt{\tfrac12 n\log n}.
    \]
\end{proof}

\subsection*{AI usage}
The authors derived the mathematics and wrote a complete draft of the paper, including all major simplification ideas. ChatGPT 5.6 plus was
then used to provide editorial suggestions and feedback on the presentation.

\bibliographystyle{abbrvurl}
\bibliography{refs.bib}

\end{document}